\documentclass[11pt]{amsart}
\usepackage{amssymb}
\usepackage{amssymb}

\begin{document}
\newcommand{\dbar}{\ensuremath{\overline\partial}}
\newcommand{\dbarstar}{\ensuremath{\overline\partial^*}}
\newcommand{\de}{\ensuremath{\partial}}
\newcommand{\C}{\ensuremath{\mathbb{C}}}
\newcommand{\R}{\ensuremath{\mathbb{R}}}
\newcommand{\D}{\ensuremath{\mathbb{D}}}
\newcommand{\T}{\ensuremath{\mathbb{T}}}

\makeatletter
\newcommand{\sumprime}{\if@display\sideset{}{'}\sum%
            \else\sum'\fi}
\makeatother

\numberwithin{equation}{section}

\newtheorem{theorem}{Theorem}[section]
\newtheorem{proposition}[theorem]{Proposition}
\newtheorem{conjecture}[theorem]{Conjecture}
\def\theconjecture{\unskip}
\newtheorem{corollary}[theorem]{Corollary}
\newtheorem{lemma}[theorem]{Lemma}
\newtheorem{observation}[theorem]{Observation}
\theoremstyle{definition}
\newtheorem{definition}{Definition}
\numberwithin{definition}{section}
\newtheorem{remark}{Remark}
\def\theremark{\unskip}
\newtheorem{question}{Question}
\def\thequestion{\unskip}
\newtheorem{example}{Example}
\def\theexample{\unskip}
\newtheorem{problem}{Problem}

\def\vvv{\ensuremath{\mid\!\mid\!\mid}}
\def\intprod{\mathbin{\lr54}}
\def\reals{{\mathbb R}}
\def\integers{{\mathbb Z}}
\def\N{{\mathbb N}}
\def\complex{{\mathbb C}\/}
\def\P{{\mathbb P}\/}
\def\dist{\operatorname{dist}\,}
\def\spec{\operatorname{spec}\,}
\def\interior{\operatorname{int}\,}
\def\trace{\operatorname{tr}\,}
\def\cl{\operatorname{cl}\,}
\def\essspec{\operatorname{esspec}\,}
\def\range{\operatorname{\mathcal R}\,}
\def\kernel{\operatorname{\mathcal N}\,}
\def\dom{\operatorname{\mathcal D}\,}
\def\linearspan{\operatorname{span}\,}
\def\lip{\operatorname{Lip}\,}
\def\sgn{\operatorname{sgn}\,}
\def\Z{ {\mathbb Z} }
\def\e{\varepsilon}
\def\p{\partial}
\def\rp{{ ^{-1} }}
\def\Re{\operatorname{Re\,} }
\def\Im{\operatorname{Im\,} }
\def\dbarb{\bar\partial_b}
\def\eps{\varepsilon}
\def\Lip{\operatorname{Lip\,}}

\def\Hs{{\mathcal H}}
\def\E{{\mathcal E}}
\def\scriptu{{\mathcal U}}
\def\scriptr{{\mathcal R}}
\def\scripta{{\mathcal A}}
\def\scriptc{{\mathcal C}}
\def\scriptd{{\mathcal D}}
\def\scripti{{\mathcal I}}
\def\scriptk{{\mathcal K}}
\def\scripth{{\mathcal H}}
\def\scriptm{{\mathcal M}}
\def\scriptn{{\mathcal N}}
\def\scripte{{\mathcal E}}
\def\scriptt{{\mathcal T}}
\def\scriptr{{\mathcal R}}
\def\scripts{{\mathcal S}}
\def\scriptb{{\mathcal B}}
\def\scriptf{{\mathcal F}}
\def\scriptg{{\mathcal G}}
\def\scriptl{{\mathcal L}}
\def\scripto{{\mathfrak o}}
\def\scriptv{{\mathcal V}}
\def\frakg{{\mathfrak g}}
\def\frakG{{\mathfrak G}}

\def\ov{\overline}

\author{Siqi Fu}
\thanks
{The author was supported in part by NSF grant  DMS-1101678.}
\address{Department of Mathematical Sciences,
Rutgers University-Camden, Camden, NJ 08102}
\email{sfu@camden.rutgers.edu}
\title[]  
{Estimates of invariant metrics on pseudoconvex domains near boundaries with constant Levi ranks}
\maketitle

\begin{abstract} Estimates of the Bergman kernel and the Bergman and Kobayashi metrics on pseudoconvex domains
near boundaries with constant Levi ranks are given.
\bigskip


\smallskip

\end{abstract}



\section{Introduction}\label{sec:intro}

In this paper we study boundary behavior of the Bergman kernel and invariant metrics
on pseudoconvex domains near boundaries whose Levi-forms have constant ranks.  Our main theorem
can be stated as follows:

\begin{theorem}\label{th:main}  Let $\Omega\subset\subset\C^n$ be a pseudoconvex
domain and let $z^0\in b\Omega$.  Let $\delta(z)$ be the Euclidean distance to the boundary. Let $K_\Omega(z, z)$ be the Bergman kernel and let $F_\Omega(z, X)$ be either the Bergman or the Kobayashi  metric. Assume that the boundary $b\Omega$ is smooth in a neighborhood $V$ of $z^0$ and its Levi-form has constant rank  $n-l-1$, $0\le l\le n-1$, for all $z\in V\cap b\Omega$.  Then there
exist a neighborhood $W\subset\subset V$ of $z^0$  and positive constants $C_1, C_2$, and $C_3$ such that
$$
C_1^{-1} |\delta(z)|^{-(n-l+1)}\le K_\Omega(z, z)\le C_1 |\delta(z)|^{-(n-l+1)}
$$
and
$$
C_2^{-1}(M(z, X)+C_3|X|^2)\le (F_\Omega(z, X))^2\le C_2 (M(z, X)+C_3|X|^2)
$$
for all $z\in W\cap \Omega$ and $X\in T^{1, 0}_z (\Omega)$,  where
$$
M(z, X)= \frac{|L_\delta(z, X)|}{|\delta (z)|}+\frac{|\langle \de \delta (z),
X\rangle|^2}{|\delta (z)|^2}.
$$

\end{theorem}

Limiting asymptotic behavior on smoothly bounded,
strongly pseudoconvex domains in $\C^n$ was obtained by Diederich \cite{Diederich70} for the Bergman kernel and metric, and by Graham \cite{Graham75} for the Carath\'{e}odory and Kobayashi metrics. In a celebrated paper \cite{Fefferman74}, Fefferman obtained asymptotic formulas for the Bergman kernel and metric and used them to establish smooth extension of biholomorphic maps. Fefferman type asymptotic expansions for the Carath\'{e}odory and Kobayashi metrics on strongly pseudoconvex domains were given in \cite{Fu95} (see also \cite{Ma91} for related results). Estimates for the Bergman kernel and the invariant metrics,  in terms of big constants and small constants, were obtained by Catlin \cite{Catlin89} for smooth bounded pseudoconvex domains of finite type in $\C^2$ and by J.-H. Chen \cite{Chen89} and McNeal \cite{Mcneal94} for convex domains of finite type in $\C^n$.  We refer the reader to the monograph of Jarnicki and Pflug \cite{JarnickiPflug93} for extensive treatment on the subject.

Strong pseudoconvexity and Levi-flatness are in a sense at the opposite ends of pseudoconvexity. Theorem~\ref{th:main} shows that in terms of boundary behavior of the invariant metrics, these two types of domains bear  striking resemblance.  Our proof of Theorem~\ref{th:main} uses an idea of Catlin (see \cite{Catlin89}): We construct plurisubharmonic functions whose complex Hessians blow
up at the rate of $1/\delta^2(z)$ in the complex normal direction and at a rate of $|L_\delta(z, X)|/\delta(z)$ in the complex tangential directions. To estimate the Kobayashi metric from below, we also use the Sibony metric~\cite{Sibony81}.

This paper is organized as follows. In Section~\ref{sec:prelim}, we recall the necessary
definitions and basic properties of the Bergman kernel and the invariant metrics. In Section~\ref{sec:foliation},
we review the local foliation of a hypersurface with constant Levi rank and choose local holomorphic coordinates under which the defining function has a desirable form. Plurisubharmonic functions with large Hessians are constructed in Section~\ref{sec:psh}. Theorem~\ref{th:main} is proved in Section~\ref{sec:estimates}.

Throughout the paper, we will use $C$ to denote positive constants which may be different in different appearances. We will also use $f\gtrsim g$ to denote $f\ge Cg$ where $C$ is a constant independent of relevant parameters, and use $A\approx B$ to denote $A\gtrsim B$ and $B\gtrsim A$.

\section{Preliminaries}\label{sec:prelim}

Let $\Omega$ be a domain in $\C^n$ and let $\D$ be the unit disc.
Let $H(\D, \Omega)$ be the set of holomorphic maps from $\D$ into
$\Omega$. Let $z\in\Omega$ and let $X=\sum_{i=1}^{n} X_{i}\de /\de z_{i} \in T_{z}^{1,0}(\Omega)$. (We sometimes identify $T^{1,0}_z(\Omega)$ with $\C^n$ without explicit notices.) The Kobayashi metric is given by
$$
F_{\Omega}^{K}(z, X) = \inf \left\{ 1/\lambda; \enspace
\enspace f\in H(\D, \Omega), f(0)=z, f'(0)=\lambda X, \lambda > 0 \right\}.
$$
Let $A^2(\Omega)$  be the space of square-integrable holomorphic functions on $\Omega$.
The Bergman kernel (on the diagonal) and metric can be defined via the following extremal properties:
\begin{align*}
K(z,z)&=\sup\left\{ |f(z)|^2; \enspace f\in A(\Omega), \|
f\|_{L^2} \le 1\right\}\\
\intertext{and}
F^{B}_{\Omega}(z,X)&=\frac{\sup\left\{ |Xf(z)|; \enspace f\in A(\Omega), f(z)=0,
\|f\|_{L^2}\le 1\right\}}{K(z,z)^{1/2}}.
\end{align*}

Denote by $\scripts_z(\Omega)$ the class of functions $u$  defined on $\Omega$ such that:
(1)~$u$ is $C^2$ on a neighborhood of $z\in\Omega$; (2)~$0\le u\le 1$  on $\Omega$ and $u(z)=0$;
and (3)~$\log u$ is plurisubharmonic on $\Omega$. The Sibony metric \cite{Sibony81} is given by
$$
F^S_{\Omega}(z, X)=\sup\left\{ \left(L_u(z, X)\right)^{1/2};\enspace
u\in\scripts_z(\Omega)\enspace\right\}
$$
where
$$
L_u(z, X)=\sum_{j, k=1}^n \frac{\partial^2 u (z)}{\partial z_j \partial\bar z_k} X_j\overline{X}_k.
$$

The Bergman metric is a K\"{a}hler metric and the Kobayashi and Sibony metrics
 are Finsler metrics. Note that while the Kobayashi metric is always
 upper semi-continuous (\cite[Prop.~3 on p.~129]{Royden71}, the Sibony metric is not, even for a domain of holomorphy (see \cite[Example~4.2.10]{JarnickiPflug93}). Both the Kobayashi and Sibony metrics are identical to the Poincar\'{e} metric on the unit disk and
satisfy the following length decreasing property: If $\Phi\colon \Omega_1\to \Omega_2$  is a holomorphic map,
then
$$
F_{\Omega_1}(z, X)\ge F_{\Omega_2}(\Phi(z), \Phi_{*z}(X)).
$$
Furthermore, $F^S_\Omega(z, X)\le F^K_\Omega(z, X)$.

Catlin \cite{Catlin89} constructed bounded plurisubharmonic functions with large
Hessians and used them to estimate the Bergman kernel and invariant metrics for
smooth bounded pseudoconvex domains of finite type in $\C^2$.  The following theorem
was proved by Catlin (\cite[Theorem~6.1 and pp.~461-462]{Catlin89}),
using H\"{o}rmander's $L^2$--estimates for the $\bar\partial$-equation. We first fix
some notations.  Denote by $D_j^{\alpha_j}$ any mixed partial derivative in $z_j$ and $\bar
z_j$ of total order $\alpha_j$. For $\alpha=(\alpha_1, \ldots, \alpha_n)$,  write $D^\alpha\phi =D_1^{\alpha_1}\cdots D_n^{\alpha_n}\phi$.

\begin{theorem}[Catlin]\label{th:Catlin}  Let $\Omega\subset\subset \C^n$ be a smoothly
bounded pseudoconvex domain and let $\hat z \in \Omega$.   Let
$\beta_1,\beta_2,\cdots, \beta_n$  be given positive numbers. Assume that
there exists a function $\phi\in C^3(\overline{\Omega})$ such that
\begin{itemize}
\item[(1)]  $|\phi(z)|\le 1$, for $z\in \Omega$;
\item[(2)]  $\phi$ is plurisubharmonic in $\Omega$;
\item[(3)]  $P=\{ z\in \C^n;\enspace |z_j-\hat z_j|< \beta_j, 1\le j\le
n\,\}\subset\Omega$;
\item[(4)]  for all $z\in P$ and $X\in \C^n$,
$$
L_{\phi}(z, X)\gtrsim \sum_{j=1}^{n} {|X_j|^2\over \beta^2_j};
$$
\item[(5)]  for all $\alpha$ with $|\alpha|\le 3$ and $z\in P$,
$$
\left| D^\alpha \phi(z)\right| \lesssim \prod_{j=1}^{n}\beta^{-\alpha_j}.
$$
\end{itemize}
Then there exists a positive constant $C$, independent of $\hat z$,  such that
\begin{equation}\label{eq:catlin1}
 C^{-1}\prod_{j=1}^{n}\frac{1}{\beta^2_j} \le K_\Omega (\hat z, \bar{\hat z})\le C
\prod_{j=1}^{n}\frac{1}{\beta^{2}_j}
\end{equation}
and
\begin{equation}\label{eq:catlin2}
C^{-1}\left\{\sum_{j=1}^{n} {|X_j|^2\over \beta^2_j}\right\}^{1/2}\le F^B_\Omega(\hat
z, X)\le C \left\{\sum_{j=1}^{n} {|X_j|^2\over
\beta^2_j}\right\}^{1/2}.
\end{equation}
\end{theorem}

The following proposition shows that same estimates also hold for
the Kobayashi and Sibony metrics:

\begin{proposition}\label{prop:s}
Let $\Omega$ be a domain in $\C^n$ and let
$\hat z\in \Omega$.  Let $\beta_1,\beta_2,\cdots, \beta_n$  be given
positive numbers.  Assume that there exists a function $\phi\in
\scriptc^2(\Omega)$ such that
\begin{itemize}
\item[(1)]  $|\phi(z)|\le 1$, for $z\in \Omega$;
\item[(2)]  $\phi$ is plurisubharmonic in $\Omega$;
\item[(3)]  $P=\{ z\in \C^n;\enspace |z_j-\hat z_j|< \beta_j, 1\le j\le
n\,\}\subset\Omega$;
\item[(4)]  for all $z\in P$ and $X\in \C^n$,
$$
L_{\phi}(z, X)\gtrsim \sum_{j=1}^{n} {|X_j|^2\over \beta^2_j}.
$$
\end{itemize}
Then there exists a positive constant $C$, independent of $\hat z$,  such
that
\begin{equation}\label{eq:s}
C^{-1}\left\{\sum_{j=1}^{n} {|X_j|^2\over \beta^2_j}\right\}^{1/2}\le F_\Omega(\hat
z, X)\le C \left\{\sum_{j=1}^{n} {|X_j|^2\over
\beta^2_j}\right\}^{1/2}
\end{equation}
where $F_\Omega$ is either the Kobayashi or the Sibony metric.
\end{proposition}

\begin{proof}  It suffices to establish the upper bounded in \eqref{eq:s}
for the Kobayashi metric  and the lower bound for the Sibony metric.
The upper bound follows directly by comparing the Kobayashi metric on $\Omega$
with that on the polydisc $P$ and using the length decreasing property.
We now prove the lower bound for the Sibony metric,
following \cite{Sibony81}.

Let $\chi(t)\colon \R^+\to [0, 1]$ be a cut--off function such that
$\chi\in C^2(\R^+)$,  $\chi(t)=t$ for $t\in [0, 1/2]$, $\chi(t)=1$ for
$t\in [1, \infty)$, and $\chi''(t)\le 0$ for $t\in \R^+$.   It is easy to see
that there exists a positive constant $\alpha$ so that
$$
t(\log\chi)''(t) +\chi'(t)\ge -\alpha
$$
for $t\in [1/2, \ 1]$.

Let
\begin{align*}
g(z)&=\sum_{j=1}^{n} {|z_j-\hat z_j|^2\over \beta_j^2}\\
\intertext{and}
u(z)&=\chi (g(z))\exp (M(\phi(z)-1))
\end{align*}
where $M$ is a large constant to be chosen.   Then $u(\hat z)=0$ and $0\le
u\le 1$ in $\Omega$. Evidently, $\log u$ is plurisubharmonic when $g(z) <1/2$ or when $g>1$.
We now consider the case when $1/2\le g(z)\le 1$.  A simple computation yields that
\begin{align*}
L_{\log u}(z, X)&=L_{\log \chi(g)}(z, X) + M L_\phi(z, X)\notag\\
     &=(\log \chi)''(g(z))\left| \langle \de g(z), X\rangle \right|^2  +(\log \chi)'(g(z)) L_g(z, X) + M L_\phi(z, X).
\end{align*}
Since
$$
\left| \langle \de g(z), X\rangle \right|^2\le g(z)\sum_{j=1}^{n}
{|X_j|^2\over \beta^2_j}
$$
and
$$
L_g(z, X)=\sum_{j=1}^{n} {|X_j|^2\over \beta^2_j},
$$
it follows that
$$
L_{\log u}(z, X)\ge \left( MC -\alpha\right) \sum_{j=1}^{n} {|X_j|^2\over
\beta^2_j},
$$
where $C$ is a  positive constant. Choosing $M\ge \alpha/C$, we then obtain that $\log u$ is plurisubharmonic on $\Omega$.  From the definition  of the Sibony metric, we then have
\begin{align*}
\left( F^S_\Omega(\hat z, X)\right)^2 &\ge \exp(M(\phi(\hat z)-1))\sum_{j=1}^{n}
{|X_j|^2\over \beta^2_j}\\
&\ge e^{-2M}\sum_{j=1}^{n} {|X_j|^2\over \beta^2_j}.
\end{align*}
We thus conclude the proof of \eqref{eq:s}.\end{proof}

\section{Levi foliations of real hypersurfaces}\label{sec:foliation}
\bigskip
We first recall well-known facts about local foliations of hypersurfaces whose Levi form has
constant rank, following \cite{Freeman74}.  Let $M$ be a smooth real hypersurface in $\C^n$. Let $z^0\in M$ and let $r(z)$ be a local defining function of $M$ on a neighborhood $V$ of $z^0$.
The {\it Levi rank} of $M$ at $z^0$, denoted by $R(M, z^0)$,  is
the number of non-zero eigenvalues of the  Levi form
\[
L_r(z^0; X, Y)=\sum_{j, k=1}^n \frac{\partial^2 r(z^0)}{\partial z_j\partial \bar z_k}X_j\overline{Y}_k
\]
for $X, Y\in T^{1, 0}_{z^0}(M)$.  The {\it Levi nullspace} $\scriptn_{z^0}$ of $M$ at $z^0$ is given by
$$
\scriptn_{z^0}=\left\{ X\in T^{1, 0}_{z^0}(M);\enspace L_r(z^0; X, Y)=0 \text{ for all } Y\in T^{1, 0}_{z^0}(M)\right\}.
$$
Thus $R(b\Omega, z^0)=n-1-\dim_{\C} \scriptn_{z^0}$.   Note that both $R(b\Omega, z^0)$ and
$\dim_{\C} \scriptn_{z^0}$  are independent of the choices of the defining functions or
local holomorphic coordinates.   A {\it complex foliation} of (complex) codimension
$q$ of $M\cap V$ is a set $\scriptf$ of complex submanifolds
of $V$ such that there exists a smooth map $\sigma\colon V\to
\R^{2q}$ of rank $2q$ on $M$ satisfying:
\begin{itemize}
\item[(1)]  $M\cap V=\left\{ z\in V;\enspace \sigma_1(z)=0\,\right\}$;
\item[(2)]  $\scriptf=\cup \left\{ M_c;\enspace c=(c_{2}, \ldots, c_{2q})\in \R^{2q-1}\right\}$, where
  $$M_c =\left\{ z\in M\cap V;\enspace \sigma_j(z)=c_j,  2\le j\le 2q\right\}.$$
\end{itemize}
Each $M_c$ is a {\it leaf} of the foliation. Note that $\scriptf$ is a complex foliation of codimension $n-l$ of $M\cap V$
defined by $\sigma=(\sigma_1,\ldots, \sigma_{2(n-l)})$  if and only if
for each $z\in M\cap V$,
there exists a neighborhood $U\subset V$ of $z$ and
holomorphic functions $f_{n-l+1}, \ldots, f_n$  on $U$ such that
\begin{itemize}
\item[(1)]  the map $F=(\sigma_1,\ldots, \sigma_{2(n-l)}, f_{n-l+1},\ldots,
f_n)$  is a diffeomorphism from  $U$  onto an open subset
$W=I\times I'\times W''$ of $\R\times \R^{2(n-l)-1}\times \C^l$;
\item[(2)]  $G=F^{-1}$ maps $\{0\}\times I'\times W''$ onto $M\cap U$;
\item[(3)]  $G_c(\cdot)=G(0, c, \cdot)$ is holomorphic on $W''$ for each
$c\in I'$.
\end{itemize}
(See \cite[Section~2]{Freeman74}.) The following theorem is well-known (cf. \cite[Theorem 6.1]{Freeman74}):

\begin{theorem}[cf. \cite{Freeman74}]\label{th:freeman}  Let $M$ be a real hypersurface in $\C^n$.
Suppose $M$ has constant constant Levi rank $n-l-1$.  Then for each $p\in M$,
there exists a neighborhood $V$ of $p$ and a unique complex foliation of
codimension $n-l$ of $M\cap  V$  such that $\scriptn_z$ is the
complex tangent space of  leaves of the foliation.
\end{theorem}

Let $\Omega\subset\subset\C^n$ be a domain and let $z^0\in b\Omega$.  Assume that
$b\Omega$ is smooth in a neighborhood $V$ of $z^0$ and that $R(b\Omega, z)=n-l-1$ for
all $z\in b\Omega\cap V$.  Let
$$
r(z)=\begin{cases}-\dist(z, b\Omega),\quad z\in \Omega;\\
             \quad \dist(z, b\Omega), \quad z\not\in \Omega.\\
             \end{cases}
$$
After a possible shrinking of $V$,  we may assume that $r(z)\in C^{\infty}(V)$.
For $z=(z_1,\ldots, z_n)$,  we write $\tilde z =(z_2, \ldots, z_n)$, $z'=(z_2, \ldots, z_{n-l})$,
and $z''=(z_{n-l+1}, \ldots, z_n)$.  The following proposition is a variation of Lemma~3.3.2 in \cite{Catlin80}.

\begin{proposition}\label{prop:foliation} Let $\Omega\subset\subset\C^n$ be a pseudoconvex
domain and let $z^0\in b\Omega$.  Assume that $b\Omega$  is smooth in a neighborhood
$V$ of $z^0$ and that $R(b\Omega, z)=n-l-1$  for all $z\in b\Omega\cap V$.  Then
there exists  a neighborhood $U\subset\subset V$  of $z^0$  such that for
each $p\in b\Omega\cap U$,  there is a biholomorphism mapping $\zeta=\Phi_p(z)$
from $U$ onto the unit ball $B(0, 1)$ that satisfies
\begin{itemize}
\item[(1)]  $\Phi_p(p)=0$;
\item[(2)]  $\Phi_p$ depends smoothly on $p$;
\item[(3)]  $\Phi_p(b\Omega\cap U)$ has a defining function of the form
\begin{equation}\label{eq:rho1}
\rho(\zeta)=\Re \zeta_1 + \sum_{j=2}^{n-l} \lambda_j |\zeta_j|^2
+O(|\zeta'|^2\cdot |\zeta''|+|\zeta'|^3 + |\Im \zeta_1|\cdot |\zeta| )
\end{equation}
near 0,  where $\lambda_j$, $2\le j\le n-l$, are positive constants depending
smoothly on p.
\end{itemize}
\end{proposition}

\begin{proof}  By Theorem~\ref{th:freeman}, there exists a complex foliation
$\scriptf$ of codimension $n-l$ for $b\Omega\cap V$.  Assume that the foliation
$\scriptf$ is defined by $\sigma =(\sigma_1,\ldots, \sigma_{2(n-l)})$.  Then
there exist a neighborhood $U\subset\subset V$
of $z^0$ and holomorphic functions $f_{n-l-1}, \cdots, f_n$ on $U$ such that
\begin{itemize}
\item[(1)]  the map $F=(\sigma_1, \ldots, \sigma_{2(n-l)},
f_{n-l+1},\ldots, f_n)$  is a diffeomorphism from $U$ onto an open subset
$W=I\times I'\times W''$ of $\R\times\R^{2(n-l)-1}\times\C^l$;
\item[(2)] $G=F^{-1}$ maps $\{0\}\times I'\times W''$ onto
$b\Omega\cap U$;
\item[(3)]  $G_c(w'')=G(0, c, w'')$ is holomorphic on
$w''=(w_{n-l+1},\ldots, w_n)\in W''$  for each $c\in I'$.
\end{itemize}
For each $z\in V$,   we write $c_z =(\sigma_2(z),\ldots, \sigma_{2(n-l)}(z))$.
Since the rank of the matrix
$$
\left( {\de G_k\over \de w_j}(0, c_{z^0}, w'')\right)_{1\le k\le n\atop
n-l+1\le j\le n}
$$
is $l$,  we may assume without loss of generality that
$$
\det \left( {\de G_k\over \de w_j}(0, c_{z^0}, w'')\right)_{n-l+1\le k\le n\atop
n-l+1\le j\le n}\not= 0.
$$
After  possible shrinking of $U$ and $W''$,  we have
$$
\det \left( {\de G_k\over \de w_j}(0, c_{p}, w'')\right)_{n-l+1\le k\le
n\atop
n-l+1\le j\le n}\not= 0
$$
for all $p\in U\cap b\Omega$ and $w''\in W''$.  Therefore the map
$$
w''\mapsto \left(G_{n-l+1}(0, c_p, w''),\ldots, G_n(0, c_p, w'')\right)
$$
is invertible on $W''$.  Let $z''\mapsto (g_{n-l+1}(z''),\ldots, g_n(z''))$ be
its inverse.  Let  $\xi=\widehat\Psi_p(z)$ be defined by
\begin{align*}
\xi_j&=z_j-G_j(0, c_p, g_{n-l+1}(z''),\ldots, g_n(z'')), \qquad 1\le
j\le n-l, \\
\intertext{and}
\xi_j &=z_j,  \qquad n-l+1\le j\le n.
\end{align*}
Then $\widehat\Psi_p$ is a biholomorphic map from $U$ into $\C^n$.  After a
 rotation and a translation,  we may assume that $\widehat{\Psi}_p(p)=0$ and
the positive $\Re\xi_1$--axis is the outward normal direction at
$\widehat{\Psi}_p(p)$ of $\widehat{\Psi}_p(b\Omega\cap U)$.  By  the implicit function theorem,
$\widehat{\Psi}_p(b\Omega\cap U)$ is defined by
$$
\hat\rho(\xi)=\Re\xi_1+\hat f (\tilde\xi)+(\Im \xi_1)\hat g (\tilde\xi,
\Im\xi_1)
$$
for $\xi$ in a neighborhood $U_1$ of 0,  where
$$
\hat f(\tilde\xi)=O(|\tilde\xi|^2); \quad \hat g(\tilde\xi, \Im
\xi_1)=O(|\tilde\xi|+|\Im\xi_1|).
$$
From property (2) of $G$ above and the definition of
$\widehat{\Psi}_p$, we have
$$
\left\{ \xi\in U;\enspace \xi_1=\cdots =\xi_{n-l}=0\right\}\subset
\widehat{\Psi}_p(b\Omega\cap U).
$$
Therefore, $\hat f(0, \xi'')\equiv 0$.  Hence the complex Hessian of $f$ at
0 has the form
\begin{equation}\label{eq:hessian}
\left({\de \hat f(0)\over \de \xi_j\de\bar \xi_j}\right)_{2\le j, k\le
n}=\begin{pmatrix}B&A\\ \bar{A}^{\tau}&0\\\end{pmatrix}
\end{equation}
where $A$ is a $(n-l-1)\times l$ matrix and $B$ is a $(n-l-1)\times (n-l-1)$
matrix.

It follows from the pseudoconvexity of $\Omega$ that the matrix in \eqref{eq:hessian} is
positive semi-definite.  Therefore, $B$ is also positive semi-definite and
$A=0$.  After a unitary transformation in the $(z_2,\ldots, z_{n-l})$-variables,  we may assume
that
$$
B=\begin{pmatrix}\lambda_2&0&\ldots &0\\
            0&\lambda_3&\ldots&0\\
            0&0&\ldots&\lambda_{n-l}\\\end{pmatrix}
$$
where $\lambda_j$, $2\le j\le n-l$,  are positive constants.   Thus the
Taylor expansion of $\hat f(\tilde\xi)$ at 0 has the form
$$
\hat f(\tilde\xi)=\sum_{j=2}^{n-l} \lambda_j |\xi_j|^2
                        + 2\Re\sum_{2\le j, k\le n-l} {\de^2
\hat{f}(0)\over \de\xi_j\de\xi_k}\xi_j\xi_k +O(|\tilde
\xi|^3)
$$
for $\tilde\xi$ near 0.

Let $\zeta=\Psi_p(\xi)$ be defined by
\begin{align*}
\zeta_1&=\xi_1 +2\sum_{2\le j, k\le n-l} {\de^2
\hat{f}(0)\over \de\xi_j\de\xi_k}\xi_j\xi_k; \\
          \tilde\zeta &=\tilde\xi.
\end{align*}
Let $\Phi_p=\Psi_p\circ\widehat{\Psi}_p$.  Then $\Phi_p(b\Omega\cap U)$ has a defining
function $\rho(\zeta)=\hat{\rho}(\Psi_p^{-1}(\zeta))$ of the form
\begin{equation}\label{eq:rho2}
\rho(\zeta)=\Re\zeta_1+f(\tilde\zeta)+(\Im\zeta_1) g(\tilde\zeta,
\Im\zeta_1)
\end{equation}
near 0,  where
\begin{equation}\label{eq:f1}
f(\tilde\zeta)=\sum_{j=2}^{n-l} \lambda_j|\zeta_j|^2
+O(|\tilde\zeta|^3)
\end{equation}
and
\begin{equation}\label{eq:g}
g(\tilde\zeta, \Im\zeta_1)=O(|\tilde\zeta|+|\Im \zeta_1|).
\end{equation}
The Taylor expansion of $f(\tilde\zeta)$ at $(0,\zeta'')$
has the form
\begin{equation}\label{eq:f2}
\begin{aligned}
f(\tilde\zeta)&=\sum_{2\le j, k\le n-l}{\de^2 f\over
\de\zeta_j\de\bar{\zeta}_k} (0, \zeta'')\zeta_j\bar{\zeta}_k\\
              &\quad+2\Re \sum_{2\le j, k\le n-l}{\de^2 f\over
\de\zeta_j\de\zeta_k} (0, \zeta'')\zeta_j\zeta_k
+O(|\zeta'|^3)\\
\end{aligned}
\end{equation}
for $\tilde\zeta$ near 0.  Comparing \eqref{eq:f1} and \eqref{eq:f2},  we have
$$
{\de^2 f\over \de\zeta_j\de\bar{\zeta}_k} (0,
\zeta'')=\lambda_j\eps_{jk}+O(|\zeta''|)
$$
and
$$
{\de^2 f\over
\de\zeta_j\de\zeta_k} (0, \zeta'')=O(|\zeta''|)
$$
for $j, k\in \{ 2, \ldots, n-l\}$,  where $\eps_{jk}=0$ if $j\not= k$ and
$\eps_{jk}=1$ if $j=k$.   Thus it follows from \eqref{eq:f2} that
\begin{equation}\label{eq:f3}
f(\tilde\zeta)=\sum_{j=2}^{n-l} \lambda_j|\zeta_j|^2 +O(|\zeta'|^2\cdot
|\zeta''|+|\zeta'|^3 ).
\end{equation}
From \eqref{eq:rho2}, \eqref{eq:g}, and \eqref{eq:f3}, we then know that $\Phi_p(b\Omega\cap U)$ has a
defining function in the form of \eqref{eq:rho1}.\end{proof}

A boundary point $p$ of $\Omega$ is called a {\it  local weak peak point} if
there exist a neighborhood $U_p$ of $p$ and a function
$f_p$ holomorphic on $\Omega\cap U_p$ and continuous on $\overline{\Omega}\cap U_p$ such that $f_p(p)=1$, $|f_p(z)|<1$ for $z\in \Omega\cap U_p$,  and $|f_p(z)|\le 1$
for  $z\in \overline{\Omega}\cap U_p$.  The function $f_p$ is called a {\it local weak
peak function~} of $\Omega$ at $p$.

\begin{corollary}\label{cor:peak} Assume the same hypotheses as in Proposition~\ref{prop:foliation}. Then each $p\in b\Omega\cap V$ is a local weak peak point of $\Omega$.
\end{corollary}

\begin{proof}  It follows from Proposition~\ref{prop:foliation} that for any $p\in b\Omega\cap V$,  there
exist a neighborhood $\widehat U_p$ of $p$ and a biholomorphic mapping $\Phi_p$
from $\widehat U_p$ onto $B(0, 1)$ such that
$$
\Phi_p(\Omega\cap \widehat U_p)\cap B(0, \eps_0)\subset\left\{ \zeta\in B(0,
\eps_0);\enspace \Re \zeta_1 -|\Im \zeta_1|<0\right\}
$$
for a sufficiently small $\eps_0>0$.  Let
$$
h(\zeta)=\exp(-(-\zeta_1)^{2/3})
$$
where the cubic root takes the principle branch by deleting the negative $\Re
\zeta_1$--axis.  Let $U_p=\Phi^{-1}_p(B(0,\eps_0))\cap \widehat U_p$.  Then
$f_p(z)=h(\Phi_p(z))$ is a local weak peak function at $p$ defined on
$\overline{\Omega}\cap U_p$.
\end{proof}

\begin{remark} Let $\Omega$ be a pseudoconvex domain with piecewise smooth boundary such that each piece has constant Levi rank.  Then it follows from the localization and length decreasing properties
of the Kobayashi metric (\cite[p.~136]{Royden71}) and Corollary~\ref{cor:peak} that $\Omega$ is Kobayashi complete.  This also follows from Theorem~\ref{th:main} (to be proved in Section~\ref{sec:estimates}). \end{remark}

\section{Construction of plurisubharmonic functions with large Hessians}\label{sec:psh}

We now turn to the construction of bounded plurisubharmonic functions with large Hessians near a piece of boundary that has constant Levi rank on a pseudoconvex domain $\Omega$.   For $\delta, a>0$,  and $X\in \C^n$,  let
\begin{align*}
P_{\delta, a}=\bigl\{ \zeta\in\C^n;\enspace |\zeta_1|<a\delta,\enspace
&|\zeta_j|<a\delta^{1\over 2}, \enspace 2\le j\le n-l,\\
              &\qquad\qquad  |\zeta_j|<a,\enspace n-l+1\le j\le n\bigr\}
\end{align*}
and let
$$
\omega(X, \delta)={|X_1|^2\over \delta^2}+\sum_{j=2}^{n-l} {|X_j|^2\over
\delta} +\sum_{j=n-l+1}^{n} |X_j|^2.
$$

We will follow the notations in Section~\ref{sec:foliation}. For  $p\in b\Omega\cap U$,  let $\widetilde{\Omega}_p=\Phi_p(\Omega\cap U)$.  The following construction
 of plurisubharmonic functions with large Hessians plays a key role in this paper
(compare \cite[Prop.~2.1]{Catlin89}; also \cite[Prop.~7]{Sibony81}).

\begin{theorem}\label{th:psh} Assume the hypothesis of Proposition~\ref{prop:foliation}.
Let $W\subset\subset U$ be a neighborhood of $z^0$.  Then for any $p\in
b\Omega\cap W$ and any sufficiently small $\delta$,  there exists a function
$g_{p,\delta}\in C^{\infty}(\overline{\widetilde{\Omega}}_p)$  and constants $a, b, C$
and $C_\alpha$, independent of $p$ and $\delta$,  such that
\begin{itemize}
\item[(1)]\enspace $|g_{p,\delta}(z)|\le 1$,  $z\in\widetilde{\Omega}_p$;
\item[(2)]\enspace $g_{p,\delta}$ is plurisubharmonic on $\widetilde{\Omega}_p$;
\item[(3)]\enspace for $\zeta\in P_{\delta, ab}\cap\widetilde{\Omega}_p$ and $Y\in \C^n$,
\[
L_{g_{p,\delta}}(\zeta, Y)\ge \frac{1}{C} \left( \frac{|\langle \partial\rho(\zeta), Y\rangle|^2}{\delta^2}+\sum_{j=2}^{n-l} {|X_j|^2\over
\delta} +\sum_{j=n-l+1}^{n} |X_j|^2\right);
\]
\item[(4)]\enspace for $\zeta\in P_{\delta, ab}\cap \widetilde{\Omega}_p$,
\[
\left| D^{\alpha}g_{p, \delta}(\zeta)\right|\le C_\alpha
\delta^{-(\alpha_1+{1\over 2}\sum_{j=2}^{n-l} \alpha_j)}
\]
where  $D^\alpha=D_1^{\alpha_1}\cdots D_n^{\alpha_n}$.
\end{itemize}
\end{theorem}

\begin{proof}     Let
$\chi_1(t)\in C^\infty(\R)$  be a decreasing function with
$\chi_1(t)=1$ for $t<{1\over2}$  and $\chi_1(t)=0$ for $t>1$.  Let
$\phi_\delta(\zeta)$ be defined by
$$
\phi_\delta(\zeta)=\chi_1\left({1\over a^2}\omega(\zeta,
\delta)\right)
$$
and let $G_{p,\delta}$ be defined by
$$
G_{p, \delta}(\zeta)=\phi_\delta(\zeta)e^{{M\over
\delta}\rho(\zeta)}
$$
where $\rho(\zeta)$ is the defining function obtained from Proposition~\ref{prop:foliation}  and
$M$ is a large constant to be chosen.

For $Y=(Y_1,\ldots, Y_n)\in \C^n$,  let
\begin{align*}
 Y^*_1&=Y_1-\sum_{j=2}^{n} {\de\rho\over \de
\zeta_j}(\zeta)\left({\de\rho\over \de\zeta_1}(\zeta)\right)^{-1} Y_j;\\
          \widetilde{Y}^*&=\widetilde Y,
\end{align*}
and $Y^*=(Y^*_1, \widetilde{Y}^*)$.  It follows from a direct calculation that
\begin{equation}\label{eq:psh1}
\begin{aligned}
L_{G_{p,\delta}}(\zeta, Y^*)&=e^{{M\over \delta}\rho(\zeta)}\bigl[
L_{\phi_\delta}(\zeta, Y^*)+2{M\over \delta}\Re \sum_{j, k=1}^{n}
{\de\phi_\delta\over
\de\zeta_j}(\zeta)Y^*_j\overline{{\de\rho\over\de\zeta_k}(\zeta) Y^*_k}\\
 &\quad + {M\over\delta}\phi_\delta L_\rho(\zeta,
Y^*)+\left({M\over\delta}\right)^{2}\phi_\delta(\zeta)\left|\langle
\de\rho(\zeta), Y^*\rangle\right|^2\bigr].\\
\end{aligned}
\end{equation}
It follows from \eqref{eq:rho1} that when $a$ is sufficiently small,
$$
\left|{\de\rho\over \de\zeta_1}(\zeta)\right|\approx 1;\quad
\left|{\de\rho\over \de\zeta_j}(\zeta)\right|\lesssim \delta^{1\over 2},\quad  2\le
j\le n-l; \quad
\left|{\de\rho\over \de\zeta_j}(\zeta)\right|\lesssim \delta, \quad n-l+1\le
j\le n
$$
for $\zeta\in P_{\delta, a}$.  Therefore
\begin{equation}\label{eq:psh2}
\left|\langle \de\rho(\zeta),
Y^*\rangle\right|^2=\left|{\de\rho\over\de\zeta_1}(\zeta) Y_1\right|^2\approx
|Y_1|^2
\end{equation}
for $\zeta\in P_{\delta, a}$.   Furthermore,
\begin{equation}\label{eq:psh3}
L_\rho(\zeta, Y^*)\gtrsim |Y'|^2- C\left(\delta |Y''|^2+\frac{|Y_1|^2}{\delta}\right)
\end{equation}
for $\zeta\in P_{\delta, a}$.

For the second term on the right hand side of \eqref{eq:psh1},  we have
\begin{align}
\bigl|{M\over \delta}\Re \sum_{j, k=1}^{n}
{\de\phi_\delta\over
\de\zeta_j}(\zeta)Y^*_j\overline{{\de\rho\over\de\zeta_k}(\zeta)
Y^*_k}\bigr|&=\bigl| {M\over \delta}\Re \Bigl(\sum_{j=1}^{n}
{\de\phi_\delta\over \de\zeta_j}(\zeta)Y^*_j\Bigr)\cdot
\Bigl({\de\rho\over\de\zeta_1}(\zeta)Y_1\Bigr)\bigr|\notag \\
&\lesssim M^{3\over 2}{|Y_1|^2\over \delta^2}+M^{1\over 2}\left( {|Y'|^2\over
\delta}+|Y''|^2\right) \label{eq:psh4}
\end{align}
for $\zeta\in P_{\delta, a}$.   Combining \eqref{eq:psh1}--\eqref{eq:psh4},
we obtain
\begin{equation}\label{eq:psh5}
L_{G_{p,\delta}}(\zeta, Y^*)\gtrsim {|Y_1|^2\over \delta^2}+{|Y'|^2\over
\delta}-C|Y''|^2
\end{equation}
if $\phi_\delta(\zeta)>{1\over 4}$ and $M$ is sufficiently large.

Let $\chi_2(t)\in C^\infty(\R)$ be a convex increasing function such that
$\chi_2(t)=0$ for $t<{1\over 2}$  and $\chi_2(t)>0$,  $\chi'_{2}(t)>0$ for
$t>{1\over 2}$.  Let $\hat g_{p, \delta}(\zeta)$ be defined by
$$
\hat g_{p, \delta}(\zeta)=\chi_2\left(G_{p, \delta}(\zeta)\right).
$$
If $\zeta\in\widetilde{\Omega}_p$ and $G_{p,\delta}(\zeta)>{1\over 4}$,  then
$\phi_\delta(\zeta) >{1\over 4}$.  Therefore it follows from \eqref{eq:psh5} that
\begin{equation}\label{eq:psh6}
\begin{aligned}
L_{\hat g_{p, \delta}}(\zeta, Y^*)&=\chi''_2(G_{p,\delta}(\zeta))\left|\langle
\de G_{p, \delta}(\zeta), Y^*\rangle\right|^2
+\chi'_2(G_{p,\delta}(\zeta))L_{G_{p, \delta}}(\zeta, Y^*)\\
&\gtrsim \chi'_2(G_{p, \delta})\left( {|Y_1|^2\over \delta^2}+{|Y'|^2\over
\delta}-C|Y''|^2 \right).
\end{aligned}
\end{equation}
If $\zeta\in P_{\delta, ab}\cap\widetilde{\Omega}_p$ and $b$ is sufficiently small,  then
$\phi_p(\zeta)=1$ and $\rho(\zeta)>-{\delta\over 2M}$.  Thus,
$$
G_{p, \delta}(\zeta)=e^{{M\over\delta}\rho(\zeta)}>e^{-{1\over 2}}>{1\over 2}.
$$
Therefore it follows from \eqref{eq:psh6} that
$$
L_{\hat g_{p, \delta}}(\zeta, Y^*)\gtrsim \chi'_2(e^{-{1\over 2}})\left(
{|Y_1|^2\over \delta^2}+{|Y'|^2\over \delta}-C|Y''|^2 \right).
$$
Let $g_{p,\delta}(\zeta)={1\over C_1}(\hat g_{p, \delta}(\zeta)+C_2
|\zeta|^2)$.  Then when $C_1$ and $C_2$ are sufficiently large,
$g_{p,\delta}(\zeta) $ satisfies properties (1)--(4) of Theorem~\ref{th:psh}.
\end{proof}
\smallskip
\smallskip

\section{Proof of the main theorem}\label{sec:estimates}
\bigskip
We prove Theorem~\ref{th:main} in this section.  Let $p\in b\Omega\cap W$.  Following the notations of Theorem~\ref{th:psh},  we
write $\widetilde{\Omega}_p=\Phi_p(\Omega\cap U)$.   Denote
\begin{align*}
Q_{\delta, c}=\bigl\{ \zeta\in \C^n;\enspace |\zeta_1-c\delta|<c^2\delta,\enspace
&|\zeta_j|<c^2\delta^{1\over 2}, \enspace 2\le j\le n-l,\\
                         &\qquad\qquad |\zeta_j|<c^2,\enspace n-l+1\le j\le
n\bigr\}.
\end{align*}
It follows from \eqref{eq:rho1} that, when $c$ and $\delta$ are sufficiently small,
\begin{equation}\label{eq:e0}
Q_{\delta, c}\subset \widetilde{\Omega}_p\cap P_{\delta, ab}.
\end{equation}
It is easy to see that
\[
|Y_1|^2\lesssim |\langle\partial\rho, Y\rangle|^2+c\left(\delta^{\frac{1}{2}}|Y'|^2+\delta |Y''|^2\right)
\]
on $Q_{\delta, c}$.  Therefore, by choosing $c$ sufficiently small, we have
\[
L_{g_{p,\delta}}(\zeta, Y)\gtrsim \omega(Y, \delta).
\]
Applying Theorems~\ref{th:Catlin} and~\ref{th:psh},  we have
\begin{equation}\label{eq:e1}
K_{\widetilde{\Omega}_p}(\zeta_\delta, \zeta_\delta)\approx
{1\over\delta^{n-l+1}}
\end{equation}
where $\zeta_\delta =(-c\delta, 0)$.

Let $p_\delta =\Phi^{-1}_p(\zeta_\delta)$.  Since $|J\Phi_p(p_\delta)|\approx
1$,  it follows from the localization property of the Bergman kernel (see
\cite{Ohsawa81}) and \eqref{eq:e1} that
\begin{align*}
K_\Omega(p_\delta, p_\delta)&\approx K_{\Omega\cap U}(p_\delta, p_\delta)\\
                  &=K_{\widetilde{\Omega}_p}(\zeta_\delta,
\zeta_\delta)|J\Phi_p(p_\delta)|^2\\
             &\approx {1\over \delta^{n-l+1}}.
\end{align*}
when $\delta$ is sufficiently small.   By letting $p$ vary on $b\Omega\cap W$ for
a small neighborhood $W\subset\subset U$ of $z^0$ and letting $\delta$
vary in $(0, \eps_0)$ for a sufficiently small $\eps_0>0$,   we obtain the estimates
for the Bergman kernel in Theorem~\ref{th:main}.

Estimates for the Bergman and Kobayashi metrics are proved similarly as above by applying Theorem~\ref{th:psh}, Theorem~\ref{th:Catlin}, and Proposition~\ref{prop:s}.
We provide only the detail for the Kobayashi metric. Let $Y=\Phi_{p*}(X)$.  By the localization property of the Kobayashi metric (see \cite[Lemma~2]{Royden71}), we have
\begin{equation}\label{eq:e2}
\begin{aligned}
F^K_\Omega (p_\delta, X)&\approx F^K_{\Omega\cap U}(p_\delta, X)\\
&= F^K_{\widetilde{\Omega}_p}(\zeta_\delta,
Y).
\end{aligned}
\end{equation}
It follows from \eqref{eq:e0}, Proposition~\ref{prop:s}, and Theorem~\ref{th:psh} that
\begin{equation}\label{eq:e3}
\left(F^K_{\widetilde{\Omega}_p}(\zeta_\delta, Y)\right)^2\approx \omega(Y, \delta)
\end{equation}
for sufficiently small $\delta$.

From the definitions of $\Phi_p$ and $Y$, it is easy to see that
\begin{equation}\label{eq:e4}
\omega(Y, \delta)-C|X|^2\lesssim {|L_\delta(p_\delta, X)|\over |\delta(p_\delta)|}
+{|\langle \de \delta(p_\delta), X\rangle |^2\over |r(p_\delta)|^2}\lesssim\omega(Y, \delta)+C|X|^2
\end{equation}
for a sufficiently large C (cf. \cite{Fu95}). This concludes the proof of Theorem~\ref{th:main}.

\begin{remark} Theorem~\ref{th:main} gives the following characterization of Levi-flatness using the Bergman kernel and invariant metrics\footnote{For the Bergman metric, this was an open problem in a recent paper of Ohsawa \cite[Question~3]{Ohsawa10}.}: Let $\Omega\subset\subset\C^n$ be a pseudoconvex domain and let $\delta(z)$ be the Euclidean distance to $b\Omega$.  Assume that $b\Omega$ is smooth in a neighborhood $V$ of $z^0\in b\Omega$. Then $b\Omega$ is Levi-flat near $z^0$ if and only if there exists a neighborhood $W\subset\subset V$ of $z^0$ and a positive constant $C$ such that either
\eqref{eq:equiv1} or \eqref{eq:equiv2} below holds:
\begin{equation}\label{eq:equiv1}
\frac{1}{C} |\delta(z)|^{-2}\le K_\Omega(z, z)\le C |\delta(z)|^{-2};
\end{equation}
\begin{equation}\label{eq:equiv2}
\frac{1}{C}\left(\frac{|\langle \de \delta (z),
X\rangle|^2}{|\delta(z)|^2}  +|X|^2\right)\le (F_\Omega(z, X))^2\le C\left(\frac{|\langle \de \delta(z),
X\rangle|^2}{|\delta(z)|^2} + |X|^2\right)
\end{equation}
for all $z\in W\cap \Omega$ and $X\in T^{1, 0}_z (\Omega)$,  where $F_\Omega(z, X)$ is either the Bergman or the Kobayashi  metric.

The sufficiency is a special case of Theorem~\ref{th:main}. To see the necessity, one observes that if $b\Omega$ is not Levi-flat near $z^0$, then by a simple continuity and inductive argument on the Levi rank, $z^0$ is an accumulation point of boundary points $z^k$ such that $b\Omega$ has constant Levi rank $\ge 1$ near each $z^k$. We then arrive at a contradiction to Theorem~\ref{th:main}.
\end{remark}

\bigskip

\noindent{\bf Acknowledgements.}  This paper was part of the author's Ph.D. thesis at Washington University in St.Louis. The author thanks his thesis advisor Steven Krantz for kind encouragements throughout the years. He also thanks Professors Bo-Yong Chen and Mei-Chi Shaw for stimulating recent discussions on related subjects which help convince him that the results in this paper might still be of current interest.

\bibliography{survey}
\providecommand{\bysame}{\leavevmode\hbox
to3em{\hrulefill}\thinspace}

\end{document}